\documentclass{article}
\usepackage[english]{babel}
\usepackage{amsmath,amsthm,amssymb,latexsym}

\usepackage[colorlinks=true,linkcolor=blue,citecolor=blue,urlcolor=black,pdfpagelabels=false]{hyperref}

\usepackage{thmtools}
\theoremstyle{plain}
  \declaretheorem[numberwithin=section]{theorem}
  \declaretheorem[numberlike=theorem]{corollary}
  \declaretheorem[numberlike=theorem]{proposition}
  \declaretheorem[numberlike=theorem]{lemma}

\theoremstyle{definition}
  
  \declaretheorem[numberlike=theorem]{example}
  \declaretheorem[numberlike=theorem]{remark}



\newcommand{\nequiv}{\mathrel{\not\equiv}}

\begin{document}

\title{Generalized Lucas congruences and linear $p$-schemes}

\author{
  Joel A.~Henningsen
  \thanks{Department of Mathematics and Statistics, University of South Alabama
  \newline\indent\hspace{6pt}\textit{Current affiliation:} Department of Mathematics, Baylor University}
  \and
  Armin Straub
  \thanks{Department of Mathematics and Statistics, University of South Alabama
  \newline\indent\hspace{6pt}\textit{Email:} \texttt{straub@southalabama.edu}}
}

\date{November 16, 2021}

\maketitle

\begin{abstract}
  We observe that a sequence satisfies Lucas congruences modulo $p$ if and
  only if its values modulo $p$ can be described by a linear $p$-scheme, as
  introduced by Rowland and Zeilberger, with a single state. This simple
  observation suggests natural generalizations of the notion of Lucas
  congruences. To illustrate this point, we prove explicit generalized Lucas
  congruences for integer sequences that can be represented as the constant
  terms of $P (x, y)^n Q (x, y)$ where $P$ and $Q$ are certain Laurent
  polynomials.
\end{abstract}

\section{Introduction}

Throughout this paper, let $p$ be a prime. We say that a sequence $A (n)$
satisfies the {\emph{Lucas congruences}} modulo $p$ if, for all $n \in
\mathbb{Z}_{\geq 0}$,
\begin{equation}
  A (n) \equiv A (n_0) A (n_1) \cdots A (n_r) \pmod{p},
  \label{eq:lucas}
\end{equation}
where $n = n_0 + n_1 p + \cdots + n_r p^r$ is the expansion of $n$ in base $p$
(in the case $n = 0$, the right-hand side is the empty product so that
\eqref{eq:lucas} specializes to $A (0) \equiv 1$). These congruences are named
after Lucas \cite{lucas78} who showed such congruences for the binomial
coefficients. Since then, Lucas congruences have received considerable
attention in the literature and many authors have shown that certain sequences
satisfy the Lucas congruences. We refer to \cite{mcintosh-lucas},
\cite{granville-bin97}, \cite{sd-laurent09}, \cite{ms-lucascongruences},
\cite{delaygue-apery}, \cite{abd-lucas}, \cite{gorodetsky-ct} and the
references therein for recent results of a more general nature as well as more
background and details. A historical survey of Lucas-type congruences can be
found in \cite{mestrovic-lucas}.

Rowland and Zeilberger \cite{rz-cong} recently introduced the notion of
{\emph{linear $p$-schemes}} to efficiently describe certain sequences $A (n)$
modulo prime powers. Namely, given a sequence $A : \mathbb{Z}_{\geq 0}
\rightarrow \mathbb{Z}$, a linear $p$-scheme for $A (n) \pmod{p^r}$ consists of {\emph{states}} $A_0, A_1, \ldots, A_m :
\mathbb{Z}_{\geq 0} \rightarrow \mathbb{Z}$ with $A_0 (n) = A (n)$ such
that, for all $i \in \{ 0, 1, \ldots, m \}$, \ $k \in \{ 0, 1, \ldots, p - 1
\}$ and $n \geq 0$,
\begin{equation}
  A_i (p n + k) \equiv \sum_{j = 0}^m \alpha_{i, j}^{(k)} A_j (n) \pmod{p^r} \label{eq:linearscheme}
\end{equation}
for some integers $\alpha_{i, j}^{(k)}$. Note that this linear $p$-scheme has
$m + 1$ states and is fully determined by the values $\alpha_{i, j}^{(k)}$
together with the initial conditions $c_i = A_i (0)$, all modulo $p^r$.
Rowland and Zeilberger \cite{rz-cong} describe algorithms to automatically
obtain such a linear $p$-scheme, for fixed $p^r$, for any sequence $A (n)$
expressible as {\emph{constant terms}}, meaning that
\begin{equation}
  A (n) = \operatorname{ct} [P (\boldsymbol{x})^n Q (\boldsymbol{x})], \label{eq:ct:pq}
\end{equation}
where $P, Q \in \mathbb{Z} [\boldsymbol{x}^{\pm 1}]$ are Laurent polynomials in
$\boldsymbol{x}= (x_1, \ldots, x_d)$, and $\operatorname{ct} [f (\boldsymbol{x})]$ denotes
the constant term of the Laurent polynomial $f (\boldsymbol{x})$. Previously,
Rowland and Yassawi \cite{ry-diag13} had described similar algorithms for
sequences that are diagonals of multivariate rational functions. We also note
that $A (n) \pmod{p^r}$ can be described by a linear
$p$-scheme, as above, if and only if $A (n) \pmod{p^r}$ is
$p$-automatic. In general, linear $p$-schemes over a ring $R$ represent
$p$-regular sequences \cite[Theorem~2.2(d)]{as-k-regular}; if $R$ is finite
(as is the case for our considerations, since $R =\mathbb{Z}/ p^r
\mathbb{Z}$), $p$-regular sequences coincide with $p$-automatic sequences
\cite[Theorem~2.3]{as-k-regular}.

\begin{example}
  The Catalan numbers $C (n)$ have the constant term expression
  \begin{equation}
    C (n) = \operatorname{ct} [(x^{- 1} + 2 + x)^n (1 - x)] . \label{eq:catalan:ct}
  \end{equation}
  Based on \eqref{eq:catalan:ct}, the algorithm of Rowland and Zeilberger
  \cite{rz-cong} can be used to construct linear $p$-schemes that describe
  the Catalan numbers modulo any fixed prime power. For instance, we find the
  following $2$-state linear $3$-scheme describing $C (n)$ modulo $3$:
  \begin{equation*}
    \begin{array}{rll}
       A_0 (3 n) & = & A_1 (n)\\
       A_0 (3 n + 1) & = & A_1 (n)\\
       A_0 (3 n + 2) & = & A_0 (n) + A_1 (n)
     \end{array} \qquad \begin{array}{rll}
       A_1 (3 n) & = & A_1 (n)\\
       A_1 (3 n + 1) & = & 2 A_1 (n)\\
       A_1 (3 n + 2) & = & 0
     \end{array}
  \end{equation*}
  Together with the initial conditions $A_0 (0) = A_1 (0) = 1$, this scheme
  uniquely describes all the values taken by the sequences $A_0$, $A_1$ and,
  therefore, the Catalan numbers $C (n) \equiv A_0 (n)$ modulo $3$. (On the
  other hand, Example~\ref{eg:catalan:3} offers a more transparent
  characterization of the Catalan numbers modulo $3$.)
\end{example}

\begin{remark}
  \label{rk:ct}We note that all constant term representations claimed in this
  paper can be algorithmically proven using, for instance, creative
  telescoping \cite{koutschan-phd}. We refer to \cite{gorodetsky-ct} for
  worked out examples of this approach.
\end{remark}

In the following, we make the rather simple observation that a sequence
satisfies the Lucas congruences modulo $p$ if and only if its values modulo
$p$ can be encoded by a linear $p$-scheme with exactly one state.

\begin{proposition}
  \label{prop:lucas:scheme}Suppose that $A (0) = 1$. Then the sequence $A (n)$
  satisfies the Lucas congruences \eqref{eq:lucas} modulo $p$ if and only if
  the values $A (n)$ modulo $p$ can be encoded by a linear $p$-scheme with a
  single state.
\end{proposition}

\begin{proof}
  Recall that, by definition, a single-state linear $p$-scheme for $A (n)$
  modulo $p$ consists of the single state $A_0 (n) \equiv A (n) \pmod{p}$ and has the property that, for all $k \in \{ 0, 1, \ldots, p
  - 1 \}$ and $n \geq 0$,
  \begin{equation*}
    A_0 (p n + k) \equiv \alpha_{0, 0}^{(k)} A_0 (n) \pmod{p} .
  \end{equation*}
  Setting $n = 0$ in this relation and using that $A (0) = 1$, we see that we
  necessarily have $\alpha_{0, 0}^{(k)} = A_0 (k)$. The relation therefore
  becomes $A (p n + k) \equiv A (k) A (n) \pmod{p}$, which
  is equivalent to the Lucas congruences \eqref{eq:lucas} modulo $p$.
\end{proof}

Proposition~\ref{prop:lucas:scheme} places the notion of Lucas congruences in
a larger context. In particular, it suggests generalizations such as the
following: one might say that a sequence satisfies an order $k$ version of the
Lucas congruences if its values modulo $p$ can be encoded by a linear
$p$-scheme with $k$ states.

In Section~\ref{sec:lucasx}, to illustrate this point, we prove such
generalized Lucas congruences of order $2$ in Corollary~\ref{cor:lucasx:simpl}
(see Remark~\ref{rk:lucas:2}). We obtain these as a special case of
Theorem~\ref{thm:lucasx} which itself provides generalized Lucas congruences
of order up to $4$ for certain constant terms $A (n) = \operatorname{ct} [P (x, y)^n Q
(x, y)]$. Another way to interpret these results is as explicitly describing
linear $p$-schemes for these sequences modulo all primes $p$ at the same time
(whereas a usual application of the algorithms of Rowland and Zeilberger
\cite{rz-cong} provides such schemes for fixed $p$). In
Section~\ref{sec:catalan}, we apply Theorem~\ref{thm:lucasx} to the case of
the Catalan numbers \eqref{eq:catalan:ct}.

In preparation for proving the results of Section~\ref{sec:lucasx}, we review
generalized central trinomial numbers in Section~\ref{sec:trinomial}. First,
however, in Section~\ref{sec:lucas}, we review, and give a simple proof of,
the following fundamental result which is remarkably effective in proving that
certain sequences satisfy Lucas congruences.

\begin{theorem}
  \label{thm:lucas:intro}Let $P \in \mathbb{Z} [\boldsymbol{x}^{\pm 1}]$ be such
  that its Newton polytope has the origin as its only interior integral point.
  Then $A (n) = \operatorname{ct} [P (\boldsymbol{x})^n]$ satisfies the Lucas
  congruences \eqref{eq:lucas} for any prime $p$.
\end{theorem}

Theorem~\ref{thm:lucas:intro} is a special case of a result of Samol and van
Straten \cite{sd-laurent09} (as well as Mellit and Vlasenko
\cite{mv-laurent13}), which shows that, if the Newton polytope of $P
(\boldsymbol{x})$ has the origin as its only interior integral point, then $A
(n) = \operatorname{ct} [P (\boldsymbol{x})^n]$ satisfies the {\emph{Dwork congruences}}
\begin{equation}
  A (p^r m + n) A (\lfloor n / p \rfloor) \equiv A (p^{r - 1} m + \lfloor n /
  p \rfloor) A (n) \pmod{p^r} \label{eq:dwork}
\end{equation}
for all primes $p$ and all integers $m, n \geq 0$, $r \geq 1$. The
case $r = 1$ of these congruences is equivalent to the Lucas congruences
\eqref{eq:lucas}. As such, Theorem~\ref{thm:lucas:intro} is a weaker result
but, as we show in Section~\ref{sec:lucas}, it can be proved much more
directly. Moreover, the proof can readily be generalized in other directions.
For instance, in Theorem~\ref{thm:lucasx}, we provide generalized Lucas
congruences for certain sequences expressible as $A (n) = \operatorname{ct} [P
(\boldsymbol{x})^n Q (\boldsymbol{x})]$. It would be interesting to know whether
these sequences satisfy generalized versions of the Dwork congruences
\eqref{eq:dwork}.

We also note that Rowland and Yassawi \cite[Theorem~5.2]{ry-diag13} provide
a rather general result in the spirit of Theorem~\ref{thm:lucas:intro} for
diagonals of certain rational functions.

\section{Lucas congruences for constant terms}\label{sec:lucas}

In the sequel, we will use the vector notation $\boldsymbol{x}= (x_1, \ldots,
x_d)$ and write, for instance, $\mathbb{Z} [\boldsymbol{x}^{\pm 1}] =\mathbb{Z}
[x_1^{\pm 1}, \ldots, x_d^{\pm 1}]$ for the ring of Laurent polynomials in $d$
variables with integer coefficients. We denote monomials as
$\boldsymbol{x}^{\boldsymbol{k}} = x_1^{k_1} \cdots x_d^{k_d}$, where
$\boldsymbol{k}= (k_1, \ldots, k_d)$ is the exponent vector. If $f
(\boldsymbol{x}) = \sum a_{\boldsymbol{k}} \boldsymbol{x}^{\boldsymbol{k}}$ is a
Laurent polynomial, then $\operatorname{supp} (f) \subseteq \mathbb{Z}^d$ denotes the
support of $f$, consisting of those $\boldsymbol{k} \in \mathbb{Z}^d$ for which
$a_{\boldsymbol{k}} \neq 0$. For such $f$, we also use the common notation
$[\boldsymbol{x}^{\boldsymbol{k}}] [f (\boldsymbol{x})] = a_{\boldsymbol{k}}$. The
Newton polytope of $f$ is the convex hull of $\operatorname{supp} (f)$. We further
denote with $\Lambda_p$ the Cartier operator
\begin{equation*}
  \Lambda_p \left[ \sum_{\boldsymbol{k} \in \mathbb{Z}^d} a_{\boldsymbol{k}}
   \boldsymbol{x}^{\boldsymbol{k}} \right] = \sum_{\boldsymbol{k} \in \mathbb{Z}^d}
   a_{p\boldsymbol{k}} \boldsymbol{x}^{\boldsymbol{k}} .
\end{equation*}
In this section, we give a short proof of Theorem~\ref{thm:lucas:intro},
copied below as Theorem~\ref{thm:lucas}, which is instructive for the
generalized Lucas congruences that we consider in Section~\ref{sec:lucasx}. In
the remainder of this section, we then illustrate the versatility of
Theorem~\ref{thm:lucas} by giving several examples.

\begin{theorem}
  \label{thm:lucas}Let $P \in \mathbb{Z} [\boldsymbol{x}^{\pm 1}]$ be such that
  its Newton polytope has the origin as its only interior integral point. Then
  $A (n) = \operatorname{ct} [P (\boldsymbol{x})^n]$ satisfies the Lucas congruences
  \eqref{eq:lucas} for any prime $p$.
\end{theorem}

\begin{proof}
  Let $n \geq 0$ and $k \in \{ 0, 1, \ldots, p - 1 \}$. We have
  \begin{eqnarray}
    A (p n + k) & = & \operatorname{ct} [P (\boldsymbol{x})^{p n} P (\boldsymbol{x})^k]
    \nonumber\\
    & \equiv & \operatorname{ct} [P (\boldsymbol{x}^p)^n P (\boldsymbol{x})^k] \pmod{p} \nonumber\\
    & = & \operatorname{ct} [P (\boldsymbol{x})^n \Lambda_p [P (\boldsymbol{x})^k]], 
    \label{eq:lucas:pnk:1}
  \end{eqnarray}
  where we used that $P (\boldsymbol{x})^{p n} \equiv P (\boldsymbol{x}^p)^n$
  modulo $p$. Write $\operatorname{supp} (P) = \{ \boldsymbol{v}_1, \boldsymbol{v}_2,
  \ldots, \boldsymbol{v}_t \}$. In other words, let $\boldsymbol{v}_i$ be the
  exponents of terms of $P (\boldsymbol{x})$. Suppose that
  $c\boldsymbol{x}^{\boldsymbol{v}}$, with $c \neq 0$, is a term of $\Lambda_p [P
  (\boldsymbol{x})^k]$. This is equivalent to $c\boldsymbol{x}^{p\boldsymbol{v}}$
  being a term of $P (\boldsymbol{x})^k$, which implies that
  \begin{equation*}
    p\boldsymbol{v}= \lambda_1 \boldsymbol{v}_1 + \lambda_2 \boldsymbol{v}_2 +
     \cdots + \lambda_t \boldsymbol{v}_t
  \end{equation*}
  where $\lambda_i \in \mathbb{Z}_{\geq 0}$ and $\lambda_1 + \lambda_2 +
  \cdots + \lambda_t = k$. It follows that
  \begin{equation*}
    \boldsymbol{v}= \mu_1 \boldsymbol{v}_1 + \mu_2 \boldsymbol{v}_2 + \cdots +
     \mu_t \boldsymbol{v}_t, \quad \mu_i = \frac{\lambda_i}{p},
  \end{equation*}
  where $\mu_i \geq 0$ and $\mu_1 + \cdots + \mu_t = k / p < 1$, which
  shows that $\boldsymbol{v}$ is an interior point of the Newton polytope of $P
  (\boldsymbol{x})$. By assumption, it must be that $\boldsymbol{v}=\boldsymbol{0}$.
  Consequently, $\Lambda_p [P (\boldsymbol{x})^k] = \operatorname{ct} [P
  (\boldsymbol{x})^k]$. Combined with \eqref{eq:lucas:pnk:1}, we conclude that
  \begin{equation*}
    A (p n + k) \equiv \operatorname{ct} [P (\boldsymbol{x})^n \operatorname{ct} [P
     (\boldsymbol{x})^k]] = \operatorname{ct} [P (\boldsymbol{x})^n] \operatorname{ct} [P
     (\boldsymbol{x})^k] = A (n) A (k) \pmod{p},
  \end{equation*}
  as claimed.
\end{proof}

Since it is particularly convenient to apply in practice, we record the
following special case of Theorem~\ref{thm:lucas}.

\begin{corollary}
  \label{cor:lucas}Let $P \in \mathbb{Z} [\boldsymbol{x}^{\pm 1}]$ and suppose
  that $\operatorname{supp} (P) \subseteq \{ - 1, 0, 1 \}^d$. Then $A (n) = \operatorname{ct}
  [P (\boldsymbol{x})^n]$ satisfies the Lucas congruences \eqref{eq:lucas} for
  any prime $p$.
\end{corollary}

We note that the case $d = 2$ can also be obtained as a special case of
Theorem~\ref{thm:lucasx} offered in the next section.

\begin{example}
  It is immediate from Corollary~\ref{cor:lucas} that the central binomial
  coefficients
  \begin{equation}
    \binom{2 n}{n} = \operatorname{ct} \left[ \frac{(1 + x)^{2 n}}{x^n} \right] =
    \operatorname{ct} \left[ \left(\frac{1}{x} + 2 + x \right)^n \right]
    \label{eq:central}
  \end{equation}
  satisfy the Lucas congruences modulo any prime. In the same manner, it
  follows from Corollary~\ref{cor:lucas} that the central trinomial
  coefficients
  \begin{equation*}
    T (n) = \operatorname{ct} [(x^{- 1} + 1 + x)^n]
  \end{equation*}
  satisfy the Lucas congruences as well. This is \cite[Theorem~4.7]{ds-cong}
  and, as recorded in Corollary~\ref{cor:lucas:trinomial}, the result extends
  directly to the generalized trinomial coefficients.
\end{example}

\begin{example}
  The sequences
  \begin{equation*}
    A_s (n) = \sum_{k_1 + \cdots + k_s = n} \binom{n}{k_1, \ldots, k_s}^2
  \end{equation*}
  count abelian squares \cite{richmond-absq09}, which are strings of length
  $2 n$ where the second half is a permutation of the first, over an alphabet
  with $s$ letters. The numbers $A_s (n)$ are also the $2 n$-th moments of the
  distance travelled by an $s$-step random walk in the plane
  \cite{bnsw-randomwalkintegrals}, where each step is of unit length and
  taken in a uniformly chosen random direction. As observed in
  \cite[(8)]{bnsw-randomwalkintegrals}, we have the constant term
  representation
  \begin{eqnarray*}
    A_s (n) & = & \operatorname{ct} [((x_1 + \cdots + x_s) (x_1^{- 1} + \cdots +
    x_s^{- 1}))^n]\\
    & = & \operatorname{ct} [((1 + x_1 + \cdots + x_{s - 1}) (1 + x_1^{- 1} + \cdots
    + x_{s - 1}^{- 1}))^n] .
  \end{eqnarray*}
  Applying Corollary~\ref{cor:lucas} to these constant terms, we are able to
  conclude that, for each positive integer $s$, the sequence $A_s (n)$
  satisfies the Lucas congruences modulo any prime.
\end{example}

\begin{example}
  \label{eg:apery}Based on the binomial sum representation
  \begin{equation}
    A (n) = \sum_{k = 0}^n \binom{n}{k}^2 \binom{n + k}{k}^2,
    \label{eq:apery3}
  \end{equation}
  Gessel \cite[Theorem 1]{gessel-super} proved that the Ap\'ery numbers $A
  (n)$ satisfy the Lucas congruences. This result can also be obtained as an
  immediate consequence of Corollary~\ref{cor:lucas} in light of the constant
  term representation \cite[Remark~1.4]{s-apery}
  \begin{equation*}
    A (n) = \operatorname{ct} \left[ \frac{(x + y) (z + 1) (x + y + z) (y + z + 1)}{x
     y z} \right]^n .
  \end{equation*}
\end{example}

We note that the Ap\'ery numbers \eqref{eq:apery3} were introduced by
R.~Ap\'ery in his surprising proof \cite{apery}, \cite{alf} of the
irrationality of $\zeta (3)$. One of their, at the time unexpected, properties
is that they satisfy a certain type of three-term recurrence. It remains an
open problem to classify the integer sequences which satisfy recurrences of
this shape. It is believed that, essentially, there are only finitely many
such sequences and, presently, $15$ such sporadic Ap\'ery-like sequences
have been found by Zagier \cite{zagier4}, Almkvist, Zudilin \cite{az-de06}
and Cooper \cite{cooper-sporadic} in extensive computer searches. Malik and
the second author \cite{ms-lucascongruences} proved that all of these $15$
sequences satisfy Lucas congruences. In $13$ of these cases, they were able to
follow McIntosh's approach \cite{mcintosh-lucas} of establishing Lucas
congruences based on suitable representations as binomials sums. On the other
hand, considerably more analysis was needed to handle the remaining two cases
(labelled $(\eta)$ and $s_{18}$). More recently, Gorodetsky
\cite{gorodetsky-ct} was able to simplify the proof of the Lucas congruences
by obtaining suitable constant term representations for each Ap\'ery-like
sporadic sequences. In $14$ cases (all except $(\eta)$), these constant term
expressions are of the form $A (n) = \operatorname{ct} [P (\boldsymbol{x})^n]$ where the
Newton polytope of $P (\boldsymbol{x})$ has the origin as its only interior
integral point. Using Theorem~\ref{thm:lucas}, it therefore follows that these
$14$ sequences satisfy the Lucas congruences.

We conclude this section by considering the sequence
\begin{equation*}
  S (n) = \sum_{k = 0}^n \binom{n}{k}^2 \binom{n - k}{k} .
\end{equation*}
Presumably based on numerical values, Z.-W.~Sun \cite[A275027]{oeis}
conjectured that $S (n) \equiv 0, \pm 1 \pmod{5}$. As
another application of Theorem~\ref{thm:lucas}, we prove this claim by showing
that the sequence $S (n)$ satisfies Lucas congruences.

\begin{lemma}
  For all $n \in \mathbb{Z}_{\geq 0}$, we have, modulo $5$,
  \begin{equation*}
    S (n) \equiv \left\{\begin{array}{ll}
       (- 1)^{d (n)}, & \text{if the digits of $n$ in base $5$ are all $0, 1,
       3$,}\\
       0, & \operatorname{otherwise},
     \end{array}\right.
  \end{equation*}
  where $d (n)$ is the number of digits of $n$ in base $5$ that are equal to
  $3$.
\end{lemma}

\begin{proof}
  We first express $S (n)$ as constant terms by following the procedure
  outlined in \cite{rz-cong} for converting certain binomial sums to
  constant terms:
  \begin{eqnarray*}
    S (n) = \sum_{k = 0}^n \binom{n}{k}^2 \binom{n - k}{k} & = & \operatorname{ct}
    \left[ \sum_{k = 0}^n \binom{n}{k} \frac{(1 + x)^n}{x^k}  \frac{(1 + y)^{n
    - k}}{y^k} \right]\\
    & = & \operatorname{ct} \left[ (1 + x)^n (1 + y)^n \left(1 + \frac{1}{x y (1 +
    y)} \right)^n \right]\\
    & = & \operatorname{ct} \left[ (1 + x)^n \left(1 + y + \frac{1}{x y} \right)^n
    \right]
  \end{eqnarray*}
  It follows immediately from this expression and Corollary~\ref{cor:lucas}
  that $S (n)$ satisfies the Lucas congruences \eqref{eq:lucas} modulo any
  prime.
  
  The claim then follows from the initial values of $S (n)$ modulo $5$:
  \begin{equation*}
    S (0) \equiv 1, \quad S (1) \equiv 1, \quad S (2) \equiv 0, \quad S (3)
     \equiv - 1, \quad S (4) \equiv 0 \pmod{5}
  \end{equation*}
\end{proof}

In \cite[A275027]{oeis} it is further observed that $S (n)$ is always odd
and that this can be seen from the alternative binomial sum
\begin{equation*}
  S (n) = \sum_{k = 0}^n \binom{n}{k} \binom{n}{2 k} \binom{2 k}{k},
\end{equation*}
combined with the fact that $\binom{2 k}{k} = 2 \binom{2 k - 1}{k - 1}$ for $k
= 1, 2, \ldots$ We note that such statements are particularly easy to
understand in terms of Lucas congruences. In this instance, it follows from $S
(0) = S (1) = 1$, together with the Lucas congruences modulo $2$, that $S (n)$
is always odd. More generally, it follows from the Lucas congruences that $S
(n)$ is never divisible by a prime $p$ if none of the values $S (0), S (1),
\ldots, S (p - 1)$ is divisible by $p$. By direct computation, this allows us
to conclude that $S (n)$ is never divisible by the following primes:
\begin{equation*}
  2, 3, 7, 11, 31, 41, 67, 73, 79, 89, 97, \ldots
\end{equation*}
It would be interesting, but appears to be much more difficult, to explicitly
characterize those primes.

\section{Generalized central trinomial numbers}\label{sec:trinomial}

Noe \cite{noe-trinomial} studied the generalized central trinomial numbers
\begin{equation*}
  T (n) = \operatorname{ct} [(a x^{- 1} + b + c x)^n] = \sum_{k = 0}^{\lfloor n / 2
   \rfloor} \binom{n}{2 k} \binom{2 k}{k} (a c)^k b^{n - 2 k},
\end{equation*}
which generalize the classical case $a = b = c = 1$ already considered by
Euler. As another immediate application of Corollary~\ref{cor:lucas}, we
obtain the following result, which is also proved by Noe
\cite[(13)]{noe-trinomial} using a congruence of Schur for Legendre
polynomials, and by Deutsch and Sagan \cite[Theorem~4.7]{ds-cong} in the
case $a = b = c = 1$.

\begin{corollary}
  \label{cor:lucas:trinomial}The generalized central trinomial numbers $T (n)$
  satisfy the Lucas congruences \eqref{eq:lucas} for any prime $p$.
\end{corollary}

Among further divisibility properties, Noe \cite[Theorem~8.8]{noe-trinomial}
determines $T (p - 1)$ modulo $p$ as follows. For $d \in \mathbb{Z}$, let $(d
/ p)$ denote the Kronecker symbol so that, for odd primes $p$, we have $(d /
p) \equiv d^{(p - 1) / 2} \pmod{p}$ while, for $p = 2$, we
simply have $(d / 2) \equiv d \pmod{2}$.

\begin{lemma}
  \label{lem:trinomial:pm1}For all primes $p$ and integers $a, b, c$,
  \begin{equation*}
    \operatorname{ct} [(a x^{- 1} + b + c x)^{p - 1}] \equiv \left(\frac{b^2 - 4 a
     c}{p} \right) \pmod{p} .
  \end{equation*}
\end{lemma}

\begin{proof}
  For odd primes, Noe \cite[Theorem~8.8]{noe-trinomial} deduces this result
  from the congruences
  \begin{equation*}
    T (p - k - 1) \equiv d^{(p - 1) / 2 - k} T (k) \pmod{p},
     \quad d = b^2 - 4 a c,
  \end{equation*}
  for the generalized trinomial coefficients, which in turn follow from the
  congruences $P_{p - k - 1} (x) \equiv P_k (x)$ for the corresponding
  Legendre polynomials due to Holt.
  
  On the other hand, for $p = 2$, we have $\operatorname{ct} [(a x^{- 1} + b + c x)^{p
  - 1}] = b$ as well as $b \equiv b^2 \equiv d \equiv (d / p) \pmod{2}$ so that the result is trivially true.
\end{proof}

We will also need the following variation.

\begin{lemma}
  \label{lem:trinomial:pm1x}If $p$ is an odd prime and $c \nequiv 0
  \pmod{p}$, then
  \begin{equation*}
    \operatorname{ct} [(a x^{- 1} + b + c x)^{p - 1} x] \equiv \frac{b}{2 c} \left(1
     - \left(\frac{b^2 - 4 a c}{p} \right) \right) \pmod{p} .
  \end{equation*}
  If $p = 2$ or $c \equiv 0 \pmod{p}$, then we have this
  congruence with the right-hand replaced by $- a b^{p - 2}$ (which is
  understood to be $a$ if $p = 2$).
\end{lemma}

\begin{proof}
  Note that
  \begin{equation*}
    \operatorname{ct} [(a x^{- 1} + b + c x)^n a x^{- 1}] = \operatorname{ct} [(c x + b + a
     x^{- 1})^n c x]
  \end{equation*}
  upon replacing $x$ by $a x^{- 1} / c$ (which does not affect the value of
  the constant term). Consequently,
  \begin{eqnarray*}
    T (n + 1) & = & \operatorname{ct} [(a x^{- 1} + b + c x)^n (a x^{- 1} + b + c
    x)]\\
    & = & b T (n) + 2 \operatorname{ct} [(a x^{- 1} + b + c x)^n x]
  \end{eqnarray*}
  so that, if $p \neq 2$ and $c \nequiv 0 \pmod{p}$,
  \begin{equation}
    \operatorname{ct} [(a x^{- 1} + b + c x)^n x] = \frac{1}{2 c} (T (n + 1) - b T
    (n)) . \label{eq:trinomial:x}
  \end{equation}
  In this case, the claim therefore follows from Lemma~\ref{lem:trinomial:pm1}
  combined with $T (p) \equiv 1 \pmod{p}$, which is a
  consequence of the fact that $T (n)$ satisfies the Lucas congruences, see
  Corollary~\ref{cor:lucas:trinomial}.
  
  The case $p = 2$ is trivial because $\operatorname{ct} [(a x^{- 1} + b + c x) x] =
  a$. On the other hand, if $c \equiv 0 \pmod{p}$, then
  $\operatorname{ct} [(a x^{- 1} + b + c x)^{p - 1} x] = (p - 1) a b^{p - 2} \equiv -
  a b^{p - 2} \pmod{p}$.
\end{proof}

\section{Generalized Lucas congruences}\label{sec:lucasx}

The following result, Theorem~\ref{thm:lucasx}, is the main technical result
of this paper and provides generalized Lucas congruences \eqref{eq:lucasx} for
certain constant terms $A (n) = \operatorname{ct} [P (x, y)^n Q (x, y)]$. The result
is somewhat involved to state in full generality but a simpler special case is
spelled out in Corollary~\ref{cor:lucasx:simpl} below. Note that it follows
from Corollary~\ref{cor:lucas} that the sequence $B (n)$ in
Theorem~\ref{thm:lucasx} satisfies the ordinary Lucas congruences
\eqref{eq:lucas}, while the sequence $\tilde{A} (n)$ in \eqref{eq:lucasx} is
such that Theorem~\ref{thm:lucasx} again applies to provide generalized Lucas
congruences \eqref{eq:lucasx} (with the same values for $\sigma_x$, $\sigma_y$
and $\hat{Q} (x, y)$). As a result, the congruences \eqref{eq:lucasx} are
sufficient to determine all values of $A (n)$ modulo any prime $p$ (from the
first $p$ values of each of the involved sequences).

\begin{theorem}
  \label{thm:lucasx}Let $A (n) = \operatorname{ct} [P (x, y)^n Q (x, y)]$ where $P, Q
  \in \mathbb{Z} [x^{\pm 1}, y^{\pm 1}]$ with
  \begin{equation*}
    P (x, y) = \sum_{(i, j) \in \{ - 1, 0, 1 \}^2} a_{i, j} x^i y^j, \quad Q
     (x, y) = \alpha + \beta x + \gamma y + \delta x y.
  \end{equation*}
  Then, for any $n \in \mathbb{Z}_{\geq 0}$ and $k \in \{ 0, 1, \ldots, p
  - 1 \}$,
  \begin{equation}
    A (p n + k) \equiv B (n) A (k) + \left\{\begin{array}{ll}
      0, & \text{if $k < p - 1$},\\
      \tilde{A} (n), & \text{if $k = p - 1$},
    \end{array}\right. \pmod{p} . \label{eq:lucasx}
  \end{equation}
  Here, $B (n) = \operatorname{ct} [P (x, y)^n]$ and $\tilde{A} (n) = \operatorname{ct} [P (x,
  y)^n \tilde{Q} (x, y)]$ with
  \begin{equation*}
    \tilde{Q} (x, y) = Q (\sigma_x x, \sigma_y y) - \alpha + \delta \hat{Q}
     (x, y),
  \end{equation*}
  where the quantities $\sigma_x, \sigma_y \in \{ 0, \pm 1 \}$ are given by
  \begin{equation}
    \sigma_x = \left(\frac{a_{1, 0}^2 - 4 a_{1, - 1} a_{1, 1}}{p} \right),
    \quad \sigma_y = \left(\frac{a_{0, 1}^2 - 4 a_{- 1, 1} a_{1, 1}}{p}
    \right), \label{eq:lucasx:signs}
  \end{equation}
  and
  \begin{equation*}
    \hat{Q} (x, y) = \frac{a_{1, 0}}{2 a_{1, 1}} (1 - \sigma_x) x +
     \frac{a_{0, 1}}{2 a_{1, 1}} (1 - \sigma_y) y + (1 - \sigma_x \sigma_y) x
     y
  \end{equation*}
  provided that $p$ is odd and $p \nmid a_{1, 1}$. If $p = 2$ or $p|a_{1, 1}$,
  then
  \begin{equation*}
    \hat{Q} (x, y) = - a_{1, - 1} a_{1, 0}^{p - 2} x - a_{- 1, 1} a_{0, 1}^{p
     - 2} y + (a_{1, 1} - \sigma_x \sigma_y) x y
  \end{equation*}
  with the understanding that, if $p = 2$, then $a^{p - 2} = 1$ for any
  integer $a$.
\end{theorem}

\begin{proof}
  As for \eqref{eq:lucas:pnk:1} in the beginning of the proof of
  Theorem~\ref{thm:lucas}, we find
  \begin{eqnarray}
    A (p n + k) & = & \operatorname{ct} [P (x, y)^{p n + k} Q (x, y)] \nonumber\\
    & \equiv & \operatorname{ct} [P (x^p, y^p)^n P (x, y)^k Q (x, y)] \pmod{p} \nonumber\\
    & = & \operatorname{ct} [P (x, y)^n \Lambda_p [P (x, y)^k Q (x, y)]] \nonumber\\
    & = & \operatorname{ct} [P^n \Lambda_p [P^k Q]] .  \label{eq:lucasx:Apnk}
  \end{eqnarray}
  If $k < p - 1$, then $\Lambda_p [P^k Q] = \operatorname{ct} [P^k Q]$ because the
  degree, in $x$ or $y$ or their inverses, of each term of $P^k Q$ is bounded
  by $k + 1 < p$. In that case, we thus get
  \begin{equation*}
    A (p n + k) \equiv \operatorname{ct} [P^n] \operatorname{ct} [P^k Q] = B (n) A (k) \pmod{p} .
  \end{equation*}
  In the remainder, we therefore consider the case $k = p - 1$.
  \begin{eqnarray}
    \Lambda_p [P^{p - 1} Q] & = & \operatorname{ct} [P^{p - 1} Q] + \sum_{T \in \{ x,
    y, x y \}} T \cdot [T^p] [P^{p - 1} Q] \nonumber\\
    & = & A (p - 1) + \sum_{T \in \{ x, y, x y \}} T \cdot \operatorname{ct} \left[
    \left(\frac{P}{T} \right)^{p - 1} \frac{Q}{T} \right] 
    \label{eq:lucasx:lambda}
  \end{eqnarray}
  For $T = x y$, we obtain
  \begin{equation*}
    \operatorname{ct} \left[ \left(\frac{P}{x y} \right)^{p - 1} \frac{Q}{x y}
     \right] = a_{1, 1}^{p - 1} \delta
  \end{equation*}
  because all nonconstant terms of the polynomial inside the constant term on
  the left-hand side feature $x$ and $y$ with negative exponents. Similarly,
  for $T = x$, we have
  \begin{equation*}
    \operatorname{ct} \left[ \left(\frac{P}{x} \right)^{p - 1} \frac{Q}{x} \right] =
     \operatorname{ct} [(a_{1, - 1} y^{- 1} + a_{1, 0} + a_{1, 1} y)^{p - 1} (\beta +
     \delta y)]
  \end{equation*}
  because the left-hand side features no terms involving $x$ with positive
  exponent. We evaluate the right-hand side using
  Lemmas~\ref{lem:trinomial:pm1} and \ref{lem:trinomial:pm1x}, which show that
  \begin{eqnarray*}
    \operatorname{ct} [(a_{1, - 1} y^{- 1} + a_{1, 0} + a_{1, 1} y)^{p - 1}] & \equiv
    & \left(\frac{a_{1, 0}^2 - 4 a_{1, - 1} a_{1, 1}}{p} \right) = \sigma_x
    \pmod{p},\\
    \operatorname{ct} [(a_{1, - 1} y^{- 1} + a_{1, 0} + a_{1, 1} y)^{p - 1} y] &
    \equiv & \frac{a_{1, 0}}{2 a_{1, 1}} (1 - \sigma_x) \pmod{p},
  \end{eqnarray*}
  the latter assuming that $p$ is odd and $p \nmid a_{1, 1}$ (we will make
  this assumption for the remainder of the proof; if $p = 2$ or $p|a_{1, 1}$
  then we only need to use the corresponding evaluation provided by
  Lemma~\ref{lem:trinomial:pm1x} instead). Combining these we therefore have
  \begin{equation*}
    \operatorname{ct} \left[ \left(\frac{P}{x} \right)^{p - 1} \frac{Q}{x} \right]
     \equiv \sigma_x \beta + \frac{a_{1, 0}}{2 a_{1, 1}} (1 - \sigma_x) \delta
     \pmod{p},
  \end{equation*}
  and, by the same arguments,
  \begin{equation*}
    \operatorname{ct} \left[ \left(\frac{P}{y} \right)^{p - 1} \frac{Q}{y} \right]
     \equiv \sigma_y \gamma + \frac{a_{0, 1}}{2 a_{1, 1}} (1 - \sigma_y)
     \delta \pmod{p} .
  \end{equation*}
  Using these evaluations in \eqref{eq:lucasx:lambda}, we have shown that
  \begin{equation*}
    \Lambda_p [P^{p - 1} Q] \equiv A (p - 1) + \tilde{Q} (x, y) \pmod{p},
  \end{equation*}
  which, applied to \eqref{eq:lucasx:Apnk}, allows us to conclude
  \begin{eqnarray*}
    A (p n + (p - 1)) & \equiv & \operatorname{ct} [P^n (A (p - 1) + \tilde{Q} (x,
    y))] \pmod{p}\\
    & = & B (n) A (p - 1) + \tilde{A} (n),
  \end{eqnarray*}
  as claimed in \eqref{eq:lucasx}.
\end{proof}

\begin{example}
  If $Q (x, y) = 1$ in Theorem~\ref{thm:lucasx}, then $B (n) = A (n)$ as well
  as $\tilde{A} (n) = 0$, so that the generalized Lucas congruences
  \eqref{eq:lucasx} become the familiar Lucas congruences \eqref{eq:lucas}
  (that these hold follows more easily from Corollary~\ref{cor:lucas}).
\end{example}

\begin{example}
  \label{eg:lucasx:hyp}Consider the sequence
  \begin{equation}
    A (n) = \operatorname{ct} \left[ \left(x + y + \frac{1}{x} - \frac{1}{y}
    \right)^n (1 + x + x y) \right] . \label{eq:lucasx:hyp:ct}
  \end{equation}
  One can show, for instance using creative telescoping, see
  Remark~\ref{rk:ct}, that $A (n)$ has the particularly simple closed form
  \begin{equation}
    A (n) = (- 1)^{\lfloor n / 2 \rfloor + \lfloor n / 4 \rfloor}
    \binom{n}{\lfloor n / 2 \rfloor} \binom{\lfloor n / 2 \rfloor}{\lfloor n /
    4 \rfloor} . \label{eq:lucasx:hyp}
  \end{equation}
  Applying Theorem~\ref{thm:lucasx} to $A (n)$, we obtain $\sigma_x = \sigma_y
  = 1$ and $\hat{Q} (x, y) = - x y$ so that $\tilde{Q} (x, y) = Q (x, y) - 1 -
  x y = x$ and, thus,
  \begin{equation}
    \tilde{A} (n) = \operatorname{ct} \left[ \left(x + y + \frac{1}{x} - \frac{1}{y}
    \right)^n x \right] = \left\{\begin{array}{ll}
      A (n), & \text{if $n$ odd,}\\
      0, & \text{if $n$ even,}
    \end{array}\right. \label{eq:lucasx:Qx}
  \end{equation}
  where the latter equality can again be shown automatically using creative
  telescoping. Similarly,
  \begin{equation*}
    B (n) = \operatorname{ct} \left[ \left(x + y + \frac{1}{x} - \frac{1}{y}
     \right)^n \right] = \left\{\begin{array}{ll}
       A (n), & \text{if $n \equiv 0 \pmod{4}$,}\\
       0, & \text{otherwise.}
     \end{array}\right.
  \end{equation*}
  Using the relations of the sequences $\tilde{A} (n)$ and $B (n)$ to $A (n)$,
  the generalized Lucas congruences \eqref{eq:lucasx} provided by
  Theorem~\ref{thm:lucasx} take the simplified form
  \begin{equation*}
    A (p n + k) \equiv \left\{\begin{array}{ll}
       A (n) A (k), & \text{if $n \equiv 0 \pmod{4}$},\\
       A (n), & \text{if $n \equiv 1 \pmod{2}$ and $k = p -
       1$},\\
       0, & \text{otherwise,}
     \end{array}\right. \pmod{p},
  \end{equation*}
  where $n \in \mathbb{Z}_{\geq 0}$ and $k \in \{ 0, 1, \ldots, p - 1
  \}$.
\end{example}

\begin{remark}
  Note that Theorem~\ref{thm:lucasx} is sufficient to apply to all cases of
  constant terms $A (n) = \operatorname{ct} [P (x, y)^n Q (x, y)]$ where $\operatorname{supp}
  (P), \operatorname{supp} (Q) \subseteq \{ - 1, 0, 1 \}^2$. To see this, observe
  that, for instance in the case of the monomial $Q (x, y) = 1 / (x y)$, the
  constant term can be rewritten as
  \begin{equation*}
    \operatorname{ct} [P (x, y)^n / (x y)] = \operatorname{ct} [P (1 / x, 1 / y)^n x y],
  \end{equation*}
  where the right-hand side is such that Theorem~\ref{thm:lucasx} applies. In
  the same way, we can handle the monomials $1 / x$, $1 / y$ as well as $x /
  y$ and $y / x$ and, therefore, reduce the cases $\operatorname{supp} (Q) \subseteq
  \{ - 1, 0, 1 \}^2$ to $\operatorname{supp} (Q) \subseteq \{ 0, 1 \}^2$.
\end{remark}

The following is a special case of Theorem~\ref{thm:lucasx}, for which the
generalized Lucas congruences \eqref{eq:lucasx} take the simplified form
\eqref{eq:lucasx:simpl}.

\begin{corollary}
  \label{cor:lucasx:simpl}Let $A (n) = \operatorname{ct} [P (x, y)^n Q (x, y)]$ where
  $P, Q \in \mathbb{Z} [x^{\pm 1}, y^{\pm 1}]$ with
  \begin{equation*}
    P (x, y) = \sum_{(i, j) \in \{ - 1, 0, 1 \}^2} a_{i, j} x^i y^j, \quad Q
     (x, y) = \alpha + \beta x + \gamma y + \delta x y.
  \end{equation*}
  Suppose that $\delta = 0$, or that both $p$ is odd and $p \nmid a_{1, 1}$.
  Suppose further that $\sigma_x = \sigma_y = 1$, where $\sigma_x, \sigma_y$
  are as in \eqref{eq:lucasx:signs}. Then, for any $n \in
  \mathbb{Z}_{\geq 0}$ and $k \in \{ 0, 1, \ldots, p - 1 \}$,
  \begin{equation}
    A (p n + k) \equiv B (n) A (k) + \left\{\begin{array}{ll}
      0, & \text{if $k < p - 1$},\\
      A (n) - A (0) B (n), & \text{if $k = p - 1$},
    \end{array}\right. \pmod{p} . \label{eq:lucasx:simpl}
  \end{equation}
  Here, again, $B (n) = \operatorname{ct} [P (x, y)^n]$.
\end{corollary}

\begin{proof}
  This is an immediate consequence of Theorem~\ref{thm:lucasx} because, under
  the present conditions, we have $\tilde{Q} (x, y) = Q (x, y) - \alpha$ which
  implies $\tilde{A} (n) = A (n) - \alpha B (n) = A (n) - A (0) B (n)$.
\end{proof}

\begin{remark}
  \label{rk:lucas:2}We note that the congruences \eqref{eq:lucasx:simpl},
  together with the Lucas congruences $B (p n + k) \equiv B (n) B (k)
  \pmod{p}$, form a two-state linear $p$-scheme (with the
  states $A_0 = A$, $A_1 = B$) characterizing the sequence $A (n)$ modulo $p$.
  In the sense discussed after Proposition~\ref{prop:lucas:scheme}, it is
  therefore shown by Corollary~\ref{cor:lucasx:simpl} that the sequence $A
  (n)$ satisfies an order $2$ version of the Lucas congruences.
\end{remark}

In Lemma~\ref{lem:lucas:shift} below, we observe that the congruences
\eqref{eq:lucasx:simpl} are natural consequences of the ordinary Lucas
congruences \eqref{eq:lucas} in the sense that, if $B (n)$ is a sequence
satisfying \eqref{eq:lucas}, then any linear combination $A (n) = \alpha B (n)
+ \beta B (n + 1)$ satisfies the congruences \eqref{eq:lucasx:simpl}. Certain
constant terms $A (n) = \operatorname{ct} [P (x, y)^n Q (x, y)]$ can indeed be
expressed as linear combinations of $B (n) = \operatorname{ct} [P (x, y)^n]$ and $B (n
+ 1)$ (see, for instance, Example~\ref{eg:ct:shift}), in which case
Corollary~\ref{cor:lucasx:simpl} can therefore be obtained (more simply) as a
consequence of Lemma~\ref{lem:lucas:shift}. On the other hand, as indicated by
Example~\ref{eg:ct:noshift}, this is not generally the case. It would be of
interest to fully characterize the constant terms which are linear
combinations of shifts of $B (n)$.

\begin{lemma}
  \label{lem:lucas:shift}Suppose that $B (n)$ satisfies the Lucas congruences
  \eqref{eq:lucas} modulo $p$. Then, for any $\alpha, \beta \in \mathbb{Z}$,
  $A (n) = \alpha B (n) + \beta B (n + 1)$ satisfies the congruences
  \eqref{eq:lucasx:simpl} modulo $p$.
\end{lemma}

\begin{proof}
  Suppose that $k < p - 1$. Then, using the fact that $B (n)$ satisfies the
  Lucas congruences \eqref{eq:lucas}, we have
  \begin{eqnarray*}
    A (p n + k) & = & \alpha B (p n + k) + \beta B (p n + (k + 1))\\
    & \equiv & \alpha B (n) B (k) + \beta B (n) B (k + 1) \pmod{p}\\
    & = & B (n) A (k),
  \end{eqnarray*}
  as in the congruences \eqref{eq:lucasx:simpl}. On the other hand, let $k = p
  - 1$. Then,
  \begin{eqnarray*}
    A (p n + p - 1) & = & \alpha B (p n + p - 1) + \beta B (p (n + 1))\\
    & \equiv & \alpha B (n) B (p - 1) + \beta B (n + 1) B (0) \pmod{p} .
  \end{eqnarray*}
  Using $B (0) = 1$ and $\alpha B (p - 1) = A (p - 1) - \beta B (p) \equiv A
  (p - 1) - \beta B (1)$ as well as $\beta B (n + 1) = A (n) - \alpha B (n)$,
  this implies
  \begin{eqnarray*}
    A (p n + p - 1) & \equiv & B (n) [A (p - 1) - \beta B (1)] + [A (n) -
    \alpha B (n)] \pmod{p}\\
    & = & B (n) A (p - 1) + A (n) - B (n) [\alpha + \beta B (1)]\\
    & = & B (n) A (p - 1) + A (n) - B (n) A (0),
  \end{eqnarray*}
  so that congruence \eqref{eq:lucasx:simpl} holds in this case as well.
\end{proof}

\begin{example}
  \label{eg:ct:shift}The sequence
  \begin{equation*}
    B (n) = \operatorname{ct} \left[ \left(4 + x + y + \frac{1}{x} + \frac{1}{y}
     \right)^n \right] = \sum_{k = 0}^n \binom{n}{k} \binom{2 k}{k} \binom{2
     (n - k)}{n - k}
  \end{equation*}
  is the Ap\'ery-like sequence labeled $\boldsymbol{E}$ by Zagier
  \cite{zagier4} (and $(d)$ in \cite{az-de06}). By
  Corollary~\ref{cor:lucas}, the sequence $B (n)$ satisfies the Lucas
  congruences \eqref{eq:lucas}. (We refer to the discussion after
  Example~\ref{eg:apery} for more information on Ap\'ery-like sequences.) It
  further follows from Corollary~\ref{cor:lucasx:simpl} that the sequence
  \begin{equation*}
    A (n) = \operatorname{ct} \left[ \left(4 + x + y + \frac{1}{x} + \frac{1}{y}
     \right)^n x \right]
  \end{equation*}
  with initial values $0, 1, 8, 57, 400, 2820, 20064, 144137, \ldots$
  satisfies the generalized Lucas congruences \eqref{eq:lucasx:simpl}. In this
  particular case, this can also be deduced from Lemma~\ref{lem:lucas:shift}
  because of the relation
  \begin{equation}
    A (n) = \frac{1}{4} B (n + 1) - B (n), \label{eq:lucas:shift}
  \end{equation}
  which follows from the fact that the Laurent polynomial $4 + x + y + x^{- 1}
  + y^{- 1}$ is symmetric in $x, y, x^{- 1}, y^{- 1}$, so that
  \begin{equation*}
    B (n + 1) = 4 B (n) + \sum_{T \in \{ x, y, x^{- 1}, y^{- 1} \}} \operatorname{ct}
     \left[ \left(4 + x + y + \frac{1}{x} + \frac{1}{y} \right)^n T \right] =
     4 (B (n) + A (n))
  \end{equation*}
  because the contribution from each of the four possibilities for $T$ is the
  same.
\end{example}

However, as illustrated by the next example, it is not generally the case in
Corollary~\ref{cor:lucasx:simpl} that the sequence $A (n)$ is a linear
combination of $B (n)$ and $B (n + 1)$ as in \eqref{eq:lucas:shift}.

\begin{example}
  \label{eg:ct:noshift}As a variation of the previous example, let us
  consider, for any integer $\lambda$, the sequences
  \begin{equation*}
    A (n) = \operatorname{ct} \left[ \left(\lambda + x + y + \frac{1}{x} -
     \frac{1}{y} \right)^n x \right] = \sum_{\substack{
       k = 0\\
       k \equiv 1 \; (2)
     }}^n \lambda^{n - k} \binom{n}{k} D (k),
  \end{equation*}
  where $D (n)$ is the hypergeometric term \eqref{eq:lucasx:hyp} and the
  latter equality is a consequence of \eqref{eq:lucasx:Qx} and binomially
  expanding $(\lambda + P)^n$ with $P = x + y + x^{- 1} - y^{- 1}$. As in the
  previous example, it follows from Corollary~\ref{cor:lucasx:simpl} that the
  sequence $A (n)$ satisfies the generalized Lucas congruences
  \eqref{eq:lucasx:simpl} with
  \begin{eqnarray*}
    B (n) & = & \operatorname{ct} \left[ \left(\lambda + x + y + \frac{1}{x} -
    \frac{1}{y} \right)^n \right] = \sum_{\substack{
      k = 0\\
      k \equiv 0 \; (4)
    }}^n \lambda^{n - k} \binom{n}{k} D (k)\\
    & = & \sum_{k = 0}^{\lfloor n / 4 \rfloor} (- 1)^k \lambda^{n - 4 k}
    \binom{n}{4 k} \binom{4 k}{2 k} \binom{2 k}{k} .
  \end{eqnarray*}
  In contrast to the previous example, however, the sequence $A (n)$ cannot be
  written as a linear combination of $B (n)$ and $B (n + 1)$ as in
  \eqref{eq:lucas:shift} (as can be seen, for instance, by comparing initial
  values).
\end{example}

\section{Catalan numbers}\label{sec:catalan}

As an application, we spell out the univariate special case of
Theorem~\ref{thm:lucasx} which is particularly simple. We then illustrate the
result by applying it to the Catalan numbers $C (n)$. In particular, we derive
a Lucas-like congruence for $C (n)$ modulo $p$ as a product of terms
corresponding to the $p$-adic digits of $n$.

\begin{corollary}
  \label{cor:lucasx:1}Let $A (n) = \operatorname{ct} [(a x^{- 1} + b + c x)^n (\alpha
  + \beta x)]$ where $a, b, c, \alpha, \beta \in \mathbb{Z}$ with $c \nequiv 0
  \pmod{p}$. Then the generalized Lucas congruences
  \eqref{eq:lucasx:simpl} hold with $B (n) = \operatorname{ct} [(a x^{- 1} + b + c
  x)^n]$.
\end{corollary}

\begin{proof}
  This follows directly from Theorem~\ref{thm:lucasx} since $\sigma_x = (c^2 /
  p) = 1$ in the present case so that we have, as in
  Corollary~\ref{cor:lucasx:simpl}, $\tilde{Q} (x, y) = Q (x, y) - \alpha$
  and, therefore, $\tilde{A} (n) = A (n) - A (0) B (n)$.
  
  Alternatively, for odd $p$, the result is a special case of
  Lemma~\ref{lem:lucas:shift} because, as observed in \eqref{eq:trinomial:x},
  \begin{equation*}
    \operatorname{ct} [(a x^{- 1} + b + c x)^n x] = \frac{1}{2 c} (B (n + 1) - b B
     (n)),
  \end{equation*}
  so that
  \begin{equation*}
    A (n) = \frac{\beta}{2 c} B (n + 1) + \left(\alpha - \frac{b \beta}{2 c}
     \right) B (n) .
  \end{equation*}
\end{proof}

\begin{example}
  The case $c = 0$ in Corollary~\ref{cor:lucasx:1} needs to be excluded since,
  in that case, $A (n) = \alpha b^n + a \beta n b^{n - 1}$ while $B (n) =
  b^n$. However, Theorem~\ref{thm:lucasx} still applies (now $\sigma_x = 0$ so
  that $\tilde{Q} (x, y) = 0$ and $\tilde{A} (n) = 0$) to show that, for any
  $n \in \mathbb{Z}_{\geq 0}$ and $k \in \{ 0, 1, \ldots, p - 1 \}$, we
  have the congruences
  \begin{equation*}
    A (p n + k) \equiv B (n) A (k) \pmod{p},
  \end{equation*}
  which are straightforward to verify directly using Fermat's little theorem.
\end{example}

Recall from \eqref{eq:catalan:ct} that the Catalan numbers have the constant
term expression $C (n) = \operatorname{ct} [(x^{- 1} + 2 + x)^n (1 - x)]$.

\begin{corollary}
  \label{cor:catalan}Let $C (n)$ be the Catalan numbers. Modulo $p$,
  \begin{equation*}
    C (p n + k) \equiv \left\{\begin{array}{ll}
       \binom{2 n}{n} C (k), & \text{if $k < p - 1$},\\
       - (2 n + 1) C (n), & \text{if $k = p - 1$} .
     \end{array}\right.
  \end{equation*}
\end{corollary}

\begin{proof}
  As noted in \eqref{eq:central}, $\operatorname{ct} [(x^{- 1} + 2 + x)^n]$ are the
  central binomial coefficients. Since $C (0) = 1$,
  Corollary~\ref{cor:lucasx:1} thus shows that, modulo $p$,
  \begin{equation*}
    C (p n + k) \equiv \binom{2 n}{n} C (k) + \left\{\begin{array}{ll}
       0, & \text{if $k < p - 1$},\\
       C (n) - \binom{2 n}{n}, & \text{if $k = p - 1$} .
     \end{array}\right.
  \end{equation*}
  In the case $k = p - 1$, it follows from Lemmas~\ref{lem:trinomial:pm1} and
  \ref{lem:trinomial:pm1x} that $C (p - 1) \equiv - 1 \pmod{p}$, which implies
  \begin{equation*}
    \binom{2 n}{n} C (p - 1) + C (n) - \binom{2 n}{n} \equiv C (n) - 2
     \binom{2 n}{n} = - (2 n + 1) C (n) \pmod{p},
  \end{equation*}
  as claimed.
\end{proof}

By iterating Corollary~\ref{cor:catalan} and combining it with the Lucas
congruences for the central binomial coefficients, we obtain the following
equivalent result which spells out generalized Lucas congruences for the
Catalan numbers in a form similar to~\eqref{eq:lucas}.

\begin{corollary}
  \label{cor:catalan:digits}Suppose the $p$-adic digits of $n$ are $p - 1,
  \ldots, p - 1, n_0, n_1, \ldots, n_r$ with $m$ initial digits that are $p -
  1$ and $n_0 \neq p - 1$. Then we have
  \begin{equation}
    C (n) \equiv \delta (n_0, m) C (n_0) \binom{2 n_1}{n_1} \cdots \binom{2
    n_r}{n_r} \pmod{p}, \label{eq:catalan:digits}
  \end{equation}
  where $\delta (n_0, m) = 1$ if $m = 0$ and $\delta (n_0, m) = - (2 n_0 + 1)$
  if $m \geq 1$.
\end{corollary}

Corollary~\ref{cor:catalan:digits} is not difficult to establish directly
(though we have not been able to find it explicitly stated in the literature):
in particular, the case $m = 0$ (equivalently, $n \nequiv - 1 \pmod{p}$) is obvious from the Lucas congruences for the central
binomial coefficients combined with the representation $C (n) = \frac{1}{n +
1} \binom{2 n}{n}$. On the other hand, we obtained
Corollary~\ref{cor:catalan:digits} in a simple and natural manner, as a very
special case of Theorem~\ref{thm:lucasx} which applies to many other sequences
of combinatorial and number theoretic interest.

\begin{example}
  \label{eg:catalan:3}The generalized Lucas congruences make certain
  properties of a sequence particularly transparent. For instance, for the
  Catalan numbers, Deutsch and Sagan \cite[Theorem~5.2]{ds-cong} prove that
  \begin{equation*}
    C (n) \equiv \left\{\begin{array}{ll}
       (- 1)^{\delta_3^{\ast} (n + 1)}, & \text{if $n + 1 \in T^{\ast}
       (01)$,}\\
       0, & \text{otherwise,}
     \end{array}\right. \pmod{3},
  \end{equation*}
  where $T^{\ast} (01)$ consists of those integers $m \geq 0$ whose
  ternary expansion $m = m_0 + 3 m_1 + 3^2 m_2 + \ldots$ is such that $m_i \in
  \{ 0, 1 \}$ for all $i \geq 1$ while, in terms of this expansion,
  $\delta_3^{\ast} (m)$ is the number of $m_i$ with $i \geq 1$ such that
  $m_i = 1$.
  
  We observe that this characterization of the Catalan numbers modulo $3$
  follows directly from Corollary~\ref{cor:catalan:digits}, the only work
  consisting in transcribing the notations: indeed, notice that $C (n_0) = 1$
  in \eqref{eq:catalan:digits} (because $C (0) = C (1) = 1$) while $\binom{2
  n_i}{n_i}$ is $1, - 1$, or $0$ modulo $3$ depending on whether $n_i$ is $0,
  1$, or $2$, respectively. We thus see from \eqref{eq:catalan:digits} that $C
  (n)$ is divisible by $3$ if and only if $m \geq 1$ and $n_0 = 1$ (in
  that case $\delta (n_0, m) = 0$), or if one of the $n_1, n_2, \ldots$ is
  $2$. This is equivalent to $n + 1 \not\in T^{\ast} (01)$. In the same manner,
  we can see that $C (n) \equiv (- 1)^{\delta_3^{\ast} (n + 1)} \pmod{3}$ if $n + 1 \in T^{\ast} (01)$.
\end{example}

\begin{example}
  To emphasize the point of the previous example, we use
  Corollary~\ref{cor:catalan:digits} to produce a similar result for the
  Catalan numbers modulo $5$:
  \begin{equation*}
    C (n) \equiv \left\{\begin{array}{ll}
       2^{\lambda (n)}, & \text{if $n \not\in Z$,}\\
       0, & \text{otherwise,}
     \end{array}\right. \pmod{5},
  \end{equation*}
  where, using the notation of Corollary~\ref{cor:catalan:digits}, the set $Z$
  consists of those integers $n \geq 0$ satisfying $n_0 = 3$, or $n_i \in
  \{ 3, 4 \}$ for some $i \geq 1$, or both $n_0 = 2$ and $m \geq 1$.
  The exponent $\lambda (n)$ is the number of $n_1, n_2, \ldots$ equal to $1$;
  and $\lambda (n)$ is increased by $1$ if $n_0 = 2$, or if both $n_0 = 1$ and
  $m \geq 1$, while $\lambda (n)$ is increased by $2$ if both $n_0 = 0$
  and $m \geq 1$. Though effective, we invite the reader to translate
  this description into a more pleasing form.
\end{example}

\section{Conclusion}\label{sec:conclusion}

Theorem~\ref{thm:lucasx} characterizes all sequences $A (n)$ modulo $p$ which
can be expressed as the constant terms of $P (x, y)^n Q (x, y)$ for Laurent
polynomials $P$ and $Q$ that are linear in each of $x, y, x^{- 1}, y^{- 1}$.
Though we have not pursued this line of inquiry here, one could follow the
same approach with the goal to deduce extensions of Theorem~\ref{thm:lucasx}
to more than two variables as well as to Laurent polynomials $P$ and $Q$ of
higher degree. This appears to quickly become considerably more intricate when
approached in full generality. However, it is likely that one can obtain
interesting results for special families of cases.

Likewise, it would be valuable to investigate general results in the spirit of
Theorem~\ref{thm:lucasx} modulo prime powers $p^r$. We point the interested
reader to Granville's engaging account \cite{granville-bin97} for an
extension of the Lucas congruences for binomial coefficients modulo $p^r$.

In a similar direction, it would be of interest to determine whether the
sequences in Theorem~\ref{thm:lucasx} satisfy generalized versions of the
Dwork congruences \eqref{eq:dwork}.

\subsection*{Acknowledgements}

This project was initiated as part of the first author's master's thesis
\cite{henningsen-msc}, which includes Proposition~\ref{prop:lucas:scheme} as
well as Corollary~\ref{cor:catalan}, under the second author's guidance. The
first author gratefully acknowledges summer support through a Sandra McLaurin
Graduate Fellowship, and the second author is grateful for support through a
Collaboration Grant (\#514645) awarded by the Simons Foundation. The authors
thank Eric Rowland for helpful comments and, in particular, for pointing out
the notion of $p$-regular sequences \cite{as-k-regular}.


\end{document}